\documentclass[10pt]{ieeeconf}

\IEEEoverridecommandlockouts

\usepackage{amsmath,amssymb}

\title{\LARGE \bf Fractional Noether's Theorem with Classical\\
and Riemann--Liouville Derivatives}

\author{Gast\~{a}o S. F. Frederico$^{1, 2}$ and Delfim F. M. Torres$^{2}$%
\thanks{$^{1}$G.S.F. Frederico is with
Greg\'{o}rio Semedo University, Luanda, Angola.}%
\thanks{$^{2}$G.S.F. Frederico and D.F.M. Torres
are with the Center for Research and Development in Mathematics and Applications,
Department of Mathematics, University of Aveiro, 3810-193 Aveiro, Portugal
{\tt\small delfim@ua.pt}}}

\newtheorem{theorem}{Theorem}
\newtheorem{corollary}[theorem]{Corollary}
\newtheorem{lemma}[theorem]{Lemma}
\newtheorem{definition}[theorem]{Definition}
\newtheorem{example}[theorem]{Example}

\newtheorem{remark}[theorem]{Remark}
\newtheorem{problem}[theorem]{Problem}

\begin{document}

\maketitle


\begin{abstract}
We prove a Noether type symmetry theorem to
fractional problems of the calculus
of variations with classical and
Riemann--Liouville derivatives. As result, we obtain
constants of motion (in the classical sense) that are valid along
the mixed classical/fractional Euler--Lagrange extremals.
Both Lagrangian and Hamiltonian versions of the Noether theorem are
obtained. Finally, we extend our Noether's theorem to more general problems
of optimal control with classical and Riemann--Liouville derivatives.
\end{abstract}





\section{INTRODUCTION}

The concept of symmetry plays an important role in science and
engineering. Symmetries are described by transformations,
which result in the same object after the transformations are carried out.
They are described mathematically by parameter groups of transformations
\cite{CD:Gel:1963,Logan:b,CD:JMS:Torres:2002a,CD:Bruce:2004}.
Their importance, as recognized by Noether in 1918 \cite{Noether:1971},
is connected with the existence of conservation laws
that can be used to reduce the order of the Euler--Lagrange differential equations
\cite{MR2351636}. Noether's symmetry theorem is nowadays recognized
as one of the most beautiful results
of the calculus of variations and optimal control
\cite{CD:Djukic:1972,CD:Kosmann:2004}.

The fractional calculus is an area of current strong
research with many different and important applications
\cite{Kilbas,CD:MilRos:1993,CD:Podlubny:1999,CD:SaKiMa:1993}.
In the last years, its importance in the calculus of variations
and optimal control has been perceived, and a fractional variational
theory began to be developed by several different authors
\cite{AGRA6,CD:BalAv:2004,Cresson,MR2349759,MR2421931,MR2905862,CD:Hilfer:2000,CD:Muslih:2005,CD:Riewe:1997,MR2895863}.
Most part of the results in this direction make use of fractional derivatives
in the sense of Riemann--Liouville
\cite{CD:BalAv:2004,MR2349759,CD:FredericoTorres:2007,CD:FredericoTorres:2006,CD:Muslih:2005},
FALVA \cite{MR2421931,MR2405377}, Caputo \cite{AGRA6,Caputo} or Riesz
\cite{CD:FredericoTorres:2010}. For a more general approach
see \cite{FVC_Gen,MyID:227}; for the state of the art
we refer the reader to the recent book \cite{book:frac}.
Here we borrow a recent ``transfer formula''
from \cite{CD:BoCrGr} that allows to obtain Noether conservation
laws, in the classical sense, for fractional variational problems.
In contrast with \cite{CD:BoCrGr}, where only momentum-type
conservation laws are obtained, here we prove a Noether-type
theorem in its general form, including both momentum and energy terms.
We trust our results will have an important impact on Mechanics.
Indeed, the classical Noether theorem is only valid for conservative systems,
while most processes observed in the physical world are
nonconservative. By dealing with Lagrangians constructed using fractional derivatives,
we can deal with nonconservative equations of motion \cite{CD:Riewe:1997}.

The article is organized as follows. In Section~\ref{sec:fdRL}
we review the basics of fractional calculus.
In Section~\ref{sec:delay} we use the Euler--Lagrange equations
(Theorem~\ref{Thm:FractELeq1}) and respective fractional extremals,
to prove the extension of Noether's theorem to fractional problems
of the calculus of variations (Theorem~\ref{theo:tndf}) and
optimal control (Theorem~\ref{thm:mainResult})
that include both classical and Riemann--Liouville derivatives.
An example of application of our main result is given in Section~\ref{exa}.


\section{PRELIMINARIES ON FRACTIONAL CALCULUS}
\label{sec:fdRL}

In this section we fix notations by collecting the definitions
and properties of fractional integrals and derivatives needed in the sequel
\cite{CD:Agrawal:2007,CD:MilRos:1993,CD:Podlubny:1999,CD:SaKiMa:1993}.

\begin{definition}(Riemann--Liouville fractional integrals)
Let $f$ be a continuous function in the interval $[a,b]$. For $t \in
[a,b]$, the left Riemann--Liouville fractional integral $_aI_t^\alpha
f(t)$ and the right Riemann--Liouville fractional integral
$_tI_b^\alpha f(t)$ of order $\alpha$,  are defined by
\begin{gather*}
_aI_t^\alpha f(t) = \frac{1}{\Gamma(\alpha)}\int_a^t
(t-\theta)^{\alpha-1}f(\theta)d\theta,\\
_tI_b^\alpha f(t) = \frac{1}{\Gamma(\alpha)}\int_t^b
(\theta-t)^{\alpha-1}f(\theta)d\theta,
\end{gather*}
where $\Gamma$ is the Euler gamma function and $0 <\alpha < 1$.
\end{definition}

\begin{definition}(Fractional derivatives in the sense of
Riemann--Liouville)
Let $f$ be a continuous function in the interval
$[a,b]$. For $t \in [a,b]$, the left Riemann--Liouville fractional
derivative $_aD_t^\alpha f(t)$ and the right Riemann--Liouville
fractional derivative $_tD_b^\alpha f(t)$ of order $\alpha$ are
defined by
\begin{equation*}
\begin{split}
_aD_t^\alpha f(t)&={\frac{d}{dt}}\, {_aI_t^{1-\alpha}} f(t)\\
&= \frac{1}{\Gamma(1-\alpha)}\frac{d}{dt}
\int_a^t (t-\theta)^{-\alpha}f(\theta)d\theta
\end{split}
\end{equation*}
and
\begin{equation*}
\begin{split}
_tD_b^\alpha f(t) &= {-\frac{d}{dt}}\,{_tI_b^{1-\alpha}}f(t)\\
&= \frac{-1}{\Gamma(1-\alpha)}\frac{d}{dt}
\int_t^b(\theta -t)^{-\alpha}f(\theta)d\theta.
\end{split}
\end{equation*}
\end{definition}

\begin{theorem}
Let $f$ and $g$ be two continuous functions
on $[a,b]$. Then, for all $t \in [a,b]$, the following property holds:
$$
_aD_t^\alpha\left(f(t)+g(t)\right)
= {_aD_t^\alpha}f(t)+{_aD_t^\alpha}g(t).
$$
\end{theorem}

\begin{remark}
In general, the Riemann--Liouville fractional derivative of a constant is not
equal to zero.
\end{remark}

We now present the integration by parts formula for fractional
derivatives.

\begin{lemma}
If $f$, $g$, and the fractional derivatives
$_aD_t^\alpha g$ and $_tD_b^\alpha f$
are continuous at every point $t\in[a,b]$, then
\begin{equation}
\label{integracao:partes}
\int_{a}^{b}f(t) _aD_t^\alpha g(t)dt
=\int_a^b g(t) _tD_b^\alpha f(t)dt
\end{equation}
for any $0<\alpha<1$. Moreover,
formula \eqref{integracao:partes} is still
valid for $\alpha=1$ provided $f$ or $g$
are zero at $t=a$ and $t=b$.
\end{lemma}

The reader interested in additional background on
fractional calculus is referred to one of the many good books on the subject
\cite{book:Baleanu,CD:Hilfer:2000,Kilbas,CD:MilRos:1993,MR2768178,CD:Podlubny:1999,CD:SaKiMa:1993}.


\section{MAIN RESULTS: EULER--LAGRANGE EQUATIONS AND NOETHER'S THEOREMS
FOR VARIATIONAL PROBLEMS WITH CLASSICAL AND RIEMANN--LIOUVILLE DERIVATIVES}
\label{sec:delay}

In Section~\ref{sub1} we prove two important results for variational problems: a
necessary optimality condition of Euler--Lagrange type (Theorem~\ref{Thm:FractELeq1}) and a
Noether-type theorem (Theorem~\ref{theo:tndf}). The results are then extended in Section~\ref{foc}
to the more general setting of optimal control.


\subsection{Fractional variational problems with classical and Riemann--Liouville derivatives}
\label{sub1}

We begin by formulating the fundamental problem under investigation.

\begin{problem}
\label{Pb1}
The fractional problem of the calculus of variations
with classical and Riemann--Liouville derivatives
consists to find the stationary functions of the functional
\begin{equation}
\label{Pf}
I[q(\cdot)] = \int_a^b
L\left(t,q(t),\dot{q}(t),{_aD_t^\alpha} q(t)\right) dt
\end{equation}
subject to given $2n$ boundary conditions
$q(a)=q_{a}$ and $q(b)=q_{b}$,
where $[a,b] \subset \mathbb{R}$, $a<b$, $0 < \alpha< 1$, $\dot{q}=\frac{dq}{dt}$,
and the admissible functions $q: t \mapsto q(t)$ and the Lagrangian
$L: (t,q,v,v_l) \mapsto L(t,q,v,v_l)$ are assumed to be $C^2$:
\begin{gather*}
q(\cdot) \in C^2\left([a,b];\,\mathbb{R}^n \right)\text{;}\\
L(\cdot,\cdot,\cdot,\cdot) \in
C^2\left([a,b]\times\mathbb{R}^n\times
\mathbb{R}^n\times\mathbb{R}^n ;\,\mathbb{R}\right)\text{.}
\end{gather*}
\end{problem}

Along the work, we denote by $\partial_{i}L$ the partial derivative
of $L$ with respect to its $i$th argument, $i=1,\ldots,4$.


\subsubsection{Fractional Euler--Lagrange equations}

The Euler--Lagrange necessary optimality condition
 is central in achieving the main results of this work.
 Our results are formulated and proved using the Euler--Lagrange equations \eqref{eq:eldf}.

A variation of $q(\cdot) \in C^2\left([a,b];\,\mathbb{R}^n \right)$ is another function of
$C^2\left([a,b];\,\mathbb{R}^n \right)$ of the form $q+
\varepsilon h$, with $h(\cdot)  \in C^2\left([a,b];\,\mathbb{R}^n \right)$ such that
 $h(a)=h(b)=0$ and $\varepsilon$ a small real positive number.

\begin{definition}(Fractional extremal with classical and Riemann--Liouville derivatives).
\label{df1}
We say that $q(\cdot)$ is an extremal with classical and
Riemann--Liouville derivatives for funcional
\eqref{Pf} if for any $h\in C^2\left([a,b];\,\mathbb{R}^n \right)$
$$
\frac{d}{d\varepsilon}\left.I[q + \varepsilon
h]\right|_{\varepsilon = 0}=0.
$$
\end{definition}

We now obtain the fractional Euler--Lagrange necessary optimality condition.

\begin{theorem}(Fractional Euler--Lagrange equation).
\label{Thm:FractELeq1}
If $q(\cdot)$ is an  extremal to
Problem~\ref{Pb1}, then it satisfies the following
\emph{Euler--Lagrange equation with classical and Riemann--Liouville derivatives}:
\begin{multline}
\label{eq:eldf}
\partial_{2} L\left(t,q(t),\dot{q}(t),{_aD_t^\alpha} q(t)\right)
-\frac{d}{dt}\partial_{3} L\left(t,q(t),\dot{q}(t),{_aD_t^\alpha}
q(t)\right) \\+ {_tD_b^\alpha}\partial_{4}
L\left(t,q(t),\dot{q}(t),{_aD_t^\alpha} q(t)\right)  = 0,\quad t \in
[a,b]\,.
\end{multline}
\end{theorem}

\begin{proof}
According with Definition~\ref{df1},
a necessary condition for $q$ to be an extremal is given by
\begin{multline}
\label{pel}
\int_a^b\Bigl[\partial_{2} L\left(t,q,\dot{q},{_aD_t^\alpha}
q\right)\cdot h+\partial_{3} L\left(t,q,\dot{q},{_aD_t^\alpha}
q \right)\cdot \dot{h}\\
+\partial_{4} L\left(t,q,\dot{q},{_aD_t^\alpha}
q \right)\cdot{_aD_t^\alpha}h\Bigr] dt=0\,.
\end{multline}
Using the fact that $h(a)=h(b)=0$, and the classical
and Riemann--Liouville \eqref{integracao:partes}
integration by parts formulas in the second and third
terms of the integrand of \eqref{pel}, respectively, we obtain
\begin{multline*}
\int_a^b\Bigl[\partial_{2} L\left(t,q,\dot{q},{_aD_t^\alpha}
q\right)-\frac{d}{dt}\partial_{3} L\left(t,q,\dot{q},{_aD_t^\alpha} q\right)\\
+ {_tD_b^\alpha}\partial_{4} L\left(t,q,\dot{q},{_aD_t^\alpha}
q\right)\Bigr]\cdot h \,dt=0.
\end{multline*}
Equality \eqref{eq:eldf} follows from the application
of the fundamental lemma of the calculus of variations
(see, \textrm{e.g.}, \cite{CD:Gel:1963}).
\end{proof}


\subsubsection{Fractional Noether's theorem}

In order to prove a fractional Noether's theorem for Problem~\ref{Pb1}
we adopt a technique used in \cite{CD:FredericoTorres:2007,CD:Jost:1998}.
The proof is done in two steps: we begin by proving a Noether's theorem
without transformation of the time
(without transformation of the independent variable); then, using
a technique of time-reparametrization, we obtain Noether's
theorem in its general form.

\begin{definition}(Invariance without transforming the time).
\label{def:inv1:MR}
Functional \eqref{Pf} is invariant under an $\varepsilon$-parameter
group of infinitesimal transformations
$\bar{q}(t)= q(t) + \varepsilon\xi(t,q) + o(\varepsilon)$ if
\begin{multline}
\label{eq:invdf}
\int_{t_{a}}^{t_{b}} L\left(t,q(t),\dot{q}(t),{_aD_t^\alpha q(t)}\right) dt\\
= \int_{t_{a}}^{t_{b}} L\left(t,\bar{q}(t),\dot{\bar{q}}(t),{_aD_t^\alpha
\bar{q}(t)}\right) dt
\end{multline}
for any subinterval $[{t_{a}},{t_{b}}] \subseteq [a,b]$.
\end{definition}

The next theorem establishes a necessary condition of invariance.

\begin{theorem}(Necessary condition of invariance).
If functional \eqref{Pf} is invariant,
in the sense of Definition~\ref{def:inv1:MR}, then
\begin{multline}
\label{eq:cnsidf}
\partial_{2} L\left(t,q(t),\dot{q}(t),{_aD_t^\alpha q(t)}\right) \cdot \xi(t,q)\\
+ \partial_{3}L\left(t,q(t),\dot{q}(t),{_aD_t^\alpha q(t)}\right)\cdot \dot{\xi}(t,q)\\
+\partial_{4} L\left(t,q(t),\dot{q}(t),{_aD_t^\alpha q(t)}\right)
\cdot {_aD_t^\alpha \xi(t,q)}  = 0.
\end{multline}
\end{theorem}

\begin{proof}
Having in mind that condition \eqref{eq:invdf}
is valid for any subinterval $[{t_{a}},{t_{b}}] \subseteq [a,b]$,
we can get rid of the integral signs in \eqref{eq:invdf}.
Differentiating this condition
with respect to $\varepsilon$, substituting $\varepsilon=0$,
and using the definitions and properties
of the Riemann--Liouville fractional derivatives given in
Section~\ref{sec:fdRL}, we arrive to
\begin{multline}
\label{eq:SP}
0 = \partial_{2} L\left(t,q,\dot{q},{_aD_t^\alpha}
q\right)\cdot\xi(t,q) +\partial_{3}L\left(t,q,\dot{q},{_aD_t^\alpha}
q\right)\cdot\dot{\xi}\\
+ \partial_{4} L\left(t,q,\dot{q},{_aD_t^\alpha}
q\right)
\cdot\frac{d}{d\varepsilon} \left[\frac{1}{\Gamma(1-\alpha)}
\frac{d}{dt}\int_a^t (t-\theta)^{-\alpha}q(\theta)d\theta\right. \\
+\left.\frac{\varepsilon}{\Gamma(1-\alpha)}\frac{d}{dt}\int_a^t
(t-\theta)^{-\alpha}\xi(\theta,q)d\theta\right]_{\varepsilon=0}\,.
\end{multline}
Expression \eqref{eq:SP} is equivalent to \eqref{eq:cnsidf}.
\end{proof}

\begin{remark}
Using the Euler--Lagrange equation \eqref{eq:eldf},
the necessary condition of invariance
\eqref{eq:cnsidf} is equivalent to
\begin{multline}
\label{eq:cnsidf11}
\xi(t,q)\cdot \frac{d}{dt}\partial_{3} L\left(t,q(t),\dot{q}(t),{_aD_t^\alpha q(t)}\right)\\
+ \partial_{3}L\left(t,q(t),\dot{q}(t),{_aD_t^\alpha q(t)}\right)\cdot \dot{\xi}(t,q)\\
+ \partial_{4} L\left(t,q(t),\dot{q}(t),{_aD_t^\alpha q(t)}\right)
\cdot {_aD_t^\alpha \xi(t,q)}\\
- \xi(t,q)\cdot {_tD_b^\alpha}\partial_{4} L\left(t,q(t),\dot{q}(t),{_aD_t^\alpha q(t)}\right) = 0 \, .
\end{multline}
\end{remark}

In \cite{CD:BoCrGr}, the following theorem is proved.

\begin{theorem}(Transfer formula \cite{CD:BoCrGr}).
\label{trfo}
Consider functions $f,g\in C^{\infty}\left([a,b];\mathbb{R}^n\right)$
and assume the following condition $(\mathcal{C})$:
the sequences $\left(g^{(k)}\cdot
_aI_t^{k-\alpha}f\right)_{k\in \mathbb{N}\setminus\{0\}}$ and
$\left(f^{(k)}\cdot _tI_b^{k-\alpha}g\right)_{k\in
\mathbb{N}\setminus\{0\}}$ converge uniformly to $0$ on $[a,b]$.
Then, the following equality holds:
\begin{multline*}
g\cdot {_aD_t^\alpha}f-f\cdot{_tD_b^\alpha}g\\
=\frac{d}{dt}\left[\sum_{r=0}^{\infty}\left((-1)^{r}g^{(r)}
\cdot {_aI_t}^{r+1-\alpha}f+f^{(r)}
\cdot {_tI_b}^{r+1-\alpha}g\right)\right].
\end{multline*}
\end{theorem}

\begin{theorem}(Fractional Noether's theorem without transformation of time).
\label{theo:tnadf1}
If functional \eqref{Pf} is invariant, in the sense of
Definition~\ref{def:inv1:MR}, and functions $\xi$ and $\partial_{4} L$
satisfy condition $(\mathcal{C})$ of Theorem~\ref{trfo}, then
\begin{multline*}
\frac{d}{dt}\Biggl[\xi\cdot\partial_{3} L
+\sum_{r=0}^{\infty}\Bigl((-1)^{r}\partial_{4} L^{(r)}
\cdot {_aI_t}^{r+1-\alpha}\xi\\
+\xi^{(r)}\cdot {_tI_b}^{r+1-\alpha}\partial_{4} L\Bigr)\Biggr] = 0
\end{multline*}
along any fractional extremal with classical and
Riemann--Liouville derivatives $q(t)$, $t \in
[a,b]$ (Definition~\ref{df1}).
\end{theorem}

\begin{proof}
We combine equation \eqref{eq:cnsidf11} and Theorem~\ref{trfo}.
\end{proof}

The next definition gives a more general notion
of invariance for the integral functional \eqref{Pf}.
The main result of this section, the
Theorem~\ref{theo:tndf}, is formulated
with the help of this definition.

\begin{definition}(Invariance of \eqref{Pf}).
\label{def:invadf}
Functional \eqref{Pf} is said to be invariant
under the $\varepsilon$-parameter group of infinitesimal transformations
$\bar{t} = t + \varepsilon\tau(t,q(t)) + o(\varepsilon)$ and
$\bar{q}(t) = q(t) + \varepsilon\xi(t,q(t)) + o(\varepsilon)$ if
\begin{multline*}
\int_{t_{a}}^{t_{b}} L\left(t,q(t),\dot{q}(t),{_aD_t^\alpha q(t)}\right) dt\\
= \int_{\bar{t}(t_a)}^{\bar{t}(t_b)} L\left(\bar{t},\bar{q}(\bar{t}),\dot{\bar{q}}(\bar{t})
{_aD_{\bar{t}}^\alpha \bar{q}(\bar{t})}\right) d\bar{t}
\end{multline*}
for any subinterval $[{t_{a}},{t_{b}}] \subseteq [a,b]$.
\end{definition}

Our next result gives a general form of Noether's theorem for fractional
problems of the calculus of variations with classical
and Riemann--Liouville derivatives.

\begin{theorem}(Fractional Noether's theorem with classical and Riemann--Liouville derivatives).
\label{theo:tndf}
If functional \eqref{Pf} is invariant,
in the sense of Definition~\ref{def:invadf}, and functions
$\xi$ and $\partial_{4} L$ satisfy condition $(\mathcal{C})$
of Theorem~\ref{trfo}, then
\begin{multline}
\label{eq:LC:Frac:RL1}
\frac{d}{dt}\Biggl[\xi\cdot\partial_{3}
L+\sum_{r=0}^{\infty}\Bigl((-1)^{r}\partial_{4} L^{(r)}
\cdot {_aI_t}^{r+1-\alpha}\xi\\
+\xi^{(r)}\cdot {_tI_b}^{r+1-\alpha}\partial_{4} L\Bigr)\\
+ \tau\left(L-\dot{q}\cdot\partial_{3} L
-\alpha\partial_{4} L\cdot{_{a}D_{t}^{\alpha}}q\right)\Biggr]= 0
\end{multline}
along any fractional extremal with classical and
Riemann--Liouville derivatives $q(t)$, $t \in [a,b]$.
\end{theorem}

\begin{proof}
Our proof is an extension of the method used in \cite{CD:Jost:1998}.
For that we reparametrize the time (the independent variable $t$)
by the Lipschitz transformation
\begin{equation*}
[a,b]\ni t\longmapsto \sigma f(\lambda) \in [\sigma_{a},\sigma_{b}]
\end{equation*}
that satisfies
\begin{equation}
\label{eq:condla}
t_{\sigma}^{'} =\frac{dt(\sigma)}{d\sigma}=
f(\lambda) = 1\,\, if\,\, \lambda=0\,.
\end{equation}
Functional \eqref{Pf} is reduced, in this way,
to an autonomous functional:
\begin{multline}
\label{eq:tempo}
\bar{I}[t(\cdot),q(t(\cdot))]\\
= \int_{\sigma_{a}}^{\sigma_{b}}
L\left(t(\sigma),q(t(\sigma)),\dot{q}(t(\sigma)),
{_{\sigma_{a}}D_{t(\sigma)}^{\alpha}q(t(\sigma))}
\right)t_{\sigma}^{'} d\sigma ,
\end{multline}
where $t(\sigma_{a}) = a$ and $t(\sigma_{b}) = b$. Using the
definitions and properties of fractional derivatives given in
Section~\ref{sec:fdRL}, we get successively that
\begin{equation*}
\begin{split}
_{\sigma_{a}}&D_{t(\sigma)}^{\alpha}q(t(\sigma))\\
&=
\frac{1}{\Gamma(1-\alpha)}\frac{d}{dt(\sigma)}
\int_{\frac{a}{f(\lambda)}}^{\sigma f(\lambda)}\left({\sigma
f(\lambda)}-\theta\right)^{-\alpha}q\left(\theta f^{-1}(\lambda)\right)d\theta    \\
&=
\frac{(t_{\sigma}^{'})^{-\alpha}}{\Gamma(1-\alpha)}
\frac{d}{d\sigma}
\int_{\frac{a}{(t_{\sigma}^{'})^{2}}}^{\sigma}
(\sigma-s)^{-\alpha}q(s)ds  \\
&= (t_{\sigma}^{'})^{-\alpha}{_{\frac{a}{(t_{\sigma}^{'})^{2}}}
D_{\sigma}^{\alpha}q(\sigma)} \, .
\end{split}
\end{equation*}
We then have
\begin{equation*}
\begin{split}
&\bar{I}[t(\cdot),q(t(\cdot))] \\
&= \int_{\sigma_{a}}^{\sigma_{b}}
L\left(t(\sigma),q(t(\sigma)),\frac{q_{\sigma}^{'}}{t_{\sigma}^{'}},(t_{\sigma}^{'})^{-\alpha}{_{\frac{a}{(t_{\sigma}^{'})^{2}}}
D_{\sigma}^{\alpha}}q(\sigma)\right) t_{\sigma}^{'} d\sigma \\
&\doteq \int_{\sigma_{a}}^{\sigma_{b}}
\bar{L}_{f}\left(t(\sigma),q(t(\sigma)),q_{\sigma}^{'},t_{\sigma}^{'},{_{\frac{a}{(t_{\sigma}^{'})^{2}}}
D_\sigma^\alpha} q(t(\sigma))\right)d\sigma \\
&= \int_a^b L\left(t,q(t),\dot{q}(t),{_aD_t^\alpha} q(t)\right) dt \\
&= I[q(\cdot)] \, .
\end{split}
\end{equation*}
If the integral functional \eqref{Pf} is invariant in the sense of
Definition~\ref{def:invadf}, then the integral functional
\eqref{eq:tempo} is invariant in the sense of
Definition~\ref{def:inv1:MR}. It follows from
Theorem~\ref{theo:tnadf1} that
\begin{multline}
\label{eq:tnadf2}
\frac{d}{dt}\Biggl[\xi\cdot\partial_{3}
\bar{L}_{f}+\tau\frac{\partial}{\partial t'_\sigma} \bar{L}_{f}
+\sum_{r=0}^{\infty}\Bigl((-1)^{r}
\partial_{5} \bar{L}_{f}^{(r)}\cdot {_aI_t}^{r+1-\alpha}\xi\\
+\xi^{(r)}\cdot {_tI_b}^{r+1-\alpha}\partial_{5} \bar{L}_{f}\Bigr)\Biggr]= 0 \, .
\end{multline}
For $\lambda = 0$, the condition \eqref{eq:condla} allow us to write that
\begin{equation*}
_{\frac{a}{(t_{\sigma}^{'})^{2}}}D_\sigma^\alpha q(t(\sigma))
 = {_aD_t}^\alpha q(t)
\end{equation*}
and, therefore, we get
\begin{equation}
\label{eq:prfMR:q1}
\begin{cases}
\partial_{3}\bar{L}_{f}
=\partial_{3} L,\\
\partial_{5}\bar{L}_{f}
=\partial_{4} L,
\end{cases}
\end{equation}
and
\begin{equation}
\label{eq:prfMR:q2}
\begin{split}
&\frac{\partial}{\partial t'_\sigma} \bar{L}_{f}
= L + \partial_{3}{\bar{L}_{f}}
\cdot t_{\sigma}^{'} \frac{\partial}{\partial t_{\sigma}^{'}}
\frac{q_{\sigma}^{'}}{t_{\sigma}^{'}}+\partial_{4}{\bar{L}_{f}}\\
&\times \frac{\partial}{\partial t_{\sigma}^{'}}\left[
\frac{(t_{\sigma}^{'})^{-\alpha}}{\Gamma(n-\alpha)}\left(\frac{d}{d\sigma}\right)^{n}
\int_{\frac{a}{(t_{\sigma}^{'})^{2}}}^{\sigma}
(\sigma-s)^{n-\alpha-1}q(s)ds\right]t_{\sigma}^{'}\\
&=-\dot{q}\cdot\partial_{3} L -\alpha\partial_{4} L\cdot{_{a}D_{t}^{\alpha}}q
+ L \, .
\end{split}
\end{equation}
We obtain \eqref{eq:LC:Frac:RL1}
substituting \eqref{eq:prfMR:q1} and \eqref{eq:prfMR:q2}
into equation \eqref{eq:tnadf2}.
\end{proof}


\subsection{Fractional optimal control problems with classical and Riemann--Liouville derivatives}
\label{foc}

We now adopt the Hamiltonian formalism in order to generalize the
Noether type results found in
\cite{CD:Djukic:1972,CD:JMS:Torres:2002a} for the more general context
of fractional optimal control problems with classical and Riemann--Liouville derivatives.
For this, we make use of our Noether's Theorem~\ref{theo:tndf} and
the standard Lagrange multiplier technique (\textrm{cf.}
\cite{CD:Djukic:1972}). The fractional optimal control problem with classical
and Riemann--Liouville derivatives is introduced, without loss of
generality, in Lagrange form:
\begin{equation}
\label{eq:COA}
I[q(\cdot),u(\cdot),\mu(\cdot)]
= \int_a^b L\left(t,q(t),u(t),\mu(t)\right) dt \longrightarrow \min
\end{equation}
subject to the differential system
\begin{gather}
\label{eq:sitRL1}
\dot{q}(t)=\varphi\left(t,q(t),u(t)\right),\\
\label{eq:sitRL}
 _aD_t^\alpha
q(t)=\rho\left(t,q(t),\mu(t)\right)
\end{gather}
and initial condition
\begin{equation}
\label{eq:COIRL}
q(a)=q_a\, .
\end{equation}
The Lagrangian $L :[a,b] \times \mathbb{R}^{n}\times
\mathbb{R}^{m}\times \mathbb{R}^{d}\rightarrow \mathbb{R}$, the velocity vector
$\varphi:[a,b] \times \mathbb{R}^{n}\times \mathbb{R}^m\rightarrow
\mathbb{R}^{n}$ and the fractional velocity vector
$\rho:[a,b] \times \mathbb{R}^{n}\times \mathbb{R}^d\rightarrow
\mathbb{R}^{n}$ are assumed to be functions of class $C^{1}$ with
respect to all their arguments. We also assume, without loss of
generality, that $0<\alpha<1$. In conformity with the calculus
of variations, we are considering that the control functions
$u(\cdot)$ and $\mu(\cdot)$ take values on an open set of $\mathbb{R}^m$
and $\mathbb{R}^d$, respectively.

\begin{remark}
\label{rem:cv:pc}
The fractional functional of the calculus of variations
with classical and Riemann--Liouville derivatives \eqref{Pf}
is obtained from \eqref{eq:COA}--\eqref{eq:sitRL} choosing
$\varphi(t,q,u)=u$ and $\rho(t,q,\mu)=\mu$.
\end{remark}


\subsubsection{Fractional Pontryagin Maximum Principle}

In this subsection we prove a fractional maximum principle
with the help of the Euler--Lagrange equations \eqref{eq:eldf}.

\begin{definition}(Process with classical and Riemann--Liouville derivatives).
\label{pros}
An admissible triplet $(q(\cdot),u(\cdot),\mu(\cdot))$ that satisfies the
control system \eqref{eq:sitRL1}--\eqref{eq:sitRL} of the
optimal control problem \eqref{eq:COA}--\eqref{eq:COIRL},
$t \in [a,b]$, is said to be a \emph{process with classical and Riemann--Liouville derivatives}.
\end{definition}

We now formulate the Fractional Pontryagin Maximum Principle
for problems with classical and Riemann--Liouville derivatives.
For convenience of notation, we introduce the following operator:
$$
[q,u,\mu,p,p_{\alpha}](t) = \left(t, q(t), u(t), \mu(t), p(t), p_{\alpha}(t)\right)
$$

\begin{theorem}(Fractional Pontryagin Maximum Principle).
\label{th:AG}
If $(q(\cdot),u(\cdot),\mu(\cdot))$ is a process
for problem \eqref{eq:COA}--\eqref{eq:COIRL}, in the sense of Definition~\ref{pros},
then there exists co-vector functions $p(\cdot)\in PC^{1}([a,b];\mathbb{R}^{n})$
and $p_{\alpha}(\cdot)\in PC^{1}([a,b];\mathbb{R}^{n})$
such that for all $t\in [a,b]$ the
quadruple  $(q(\cdot),u(\cdot),p(\cdot),p_{\alpha}(\cdot))$ satisfies
the following conditions:
\begin{itemize}
\item the Hamiltonian system
\begin{equation*}
\label{eq:HamRL}
\begin{cases}
\partial_5 {\cal H}[q,u,\mu,p,p_{\alpha}](t)=\dot{q}(t)\,  ,\\
\partial_6 {\cal H}[q,u,\mu,p,p_{\alpha}](t)={_aD_t^\alpha} q(t)  \, , \\
\partial_2{\cal H}[q,u,\mu,p,p_{\alpha}](t)=-\dot{p}(t)+_tD_b^\alpha p_{\alpha}(t)  \, ;
\end{cases}
\end{equation*}
\item the stationary conditions
\begin{gather*}
 \partial_3 {\cal H}[q,u,\mu,p,p_{\alpha}](t)=0 \, ,\\
 \partial_4 {\cal H}[q,u,\mu,p,p_{\alpha}](t)=0 \, ;
\end{gather*}
\end{itemize}
where the Hamiltonian ${\cal H}$ is given by
\begin{multline}
\label{eq:HL}
{\cal H}\left(t,q,u,\mu,p,p_{\alpha}\right) \\
= L\left(t,q,u,\mu\right) + p \cdot \varphi\left(t,q,u\right)
+p_{\alpha}\cdot \rho(t,q,\mu) \, .
\end{multline}
\end{theorem}

\begin{proof}
Minimizing \eqref{eq:COA} subject to
\eqref{eq:sitRL1}--\eqref{eq:sitRL} is equivalent,
by the Lagrange multiplier rule,
to minimize
\begin{multline}
\label{eq:COA1} J[q(\cdot),u(\cdot),\mu(\cdot),p(\cdot),p_{\alpha}(\cdot)]
= \int_a^b \Bigl[{\cal H}[q,u,\mu,p,p_{\alpha}](t)\\
- p(t) \cdot \dot{q}(t)-p_{\alpha}(t)\cdot {_aD}_t^\alpha q(t)\Bigr]dt
\end{multline}
with ${\cal H}$ given by \eqref{eq:HL}.
Theorem~\ref{th:AG} is easily proved applying the
optimality condition \eqref{eq:eldf}
to the augmented functional \eqref{eq:COA1}.
\end{proof}

\begin{definition}(Pontryagin extremal with classical
and fractional derivatives).
\label{PR}
A tuple $\left(q(\cdot),u(\cdot),\mu(\cdot),p(\cdot),p_{\alpha}(\cdot)\right)$
satisfying Theorem~\ref{th:AG} is called a \emph{Pontryagin extremal
with classical and Riemann--Liouville derivatives}.
\end{definition}

\begin{remark}
For problems of the calculus of variations
with classical and Riemann--Liouville derivatives,
one has $\varphi(t,q,u)=u$ and $\rho(t,q,\mu)=\mu$ (Remark~\ref{rem:cv:pc}).
Therefore, ${\cal H} = L + p \cdot u +p_{\alpha}\cdot\mu$. From the Hamiltonian system
of Theorem~\ref{th:AG} we get
\begin{equation}
\begin{cases}\label{edr}
 u=\dot{q}\\
\mu={_aD_t^\alpha} q\\
\partial_2L=-\dot{p}+_tD_b^\alpha p_{\alpha}
\end{cases}
\end{equation}
and from the stationary conditions
\begin{gather*}
\partial_3{\cal H}=0\Leftrightarrow \partial_3 L=-p
\Rightarrow\frac{d}{dt}\partial_3 L=-\dot{p},\\
\partial_4{\cal H}=0\Leftrightarrow \partial_4 L=-p_{\alpha}
\Rightarrow {_tD_b^\alpha}\partial_4 L=-{_tD_b^\alpha} p_{\alpha}.
\end{gather*}
Substituting these two quantities into \eqref{edr}, we arrive to
the Euler--Lagrange equations with classical
and Riemann--Liouville derivatives \eqref{eq:eldf}.
\end{remark}


\subsubsection{Noether's theorem for fractional optimal control problems}

The notion of variational invariance for \eqref{eq:COA}--\eqref{eq:sitRL}
is defined with the help of the augmented functional \eqref{eq:COA1}.

\begin{definition}(Variational invariance of \eqref{eq:COA}--\eqref{eq:sitRL}).
\label{def:inv:gt1}
We say that the integral functional \eqref{eq:COA1}
is invariant under the $\varepsilon$-parameter
family of infinitesimal transformations
\begin{equation}
\label{eq:trf:inf}
\begin{cases}
\bar{t} = t+\varepsilon\tau[q,u,\mu,p,p_{\alpha}](t) + o(\varepsilon) \, , \\
\bar{q}(t) = q(t)+\varepsilon\xi[q,u,\mu,p,p_{\alpha}](t) + o(\varepsilon) \, , \\
\bar{u}(t) = u(t)+\varepsilon\varrho[q,u,\mu,p,p_{\alpha}](t) + o(\varepsilon) \, , \\
\bar{\mu}(t)=\mu(t)+\varepsilon\iota[q,u,\mu,p,p_{\alpha}](t)+o(\varepsilon)\, ,\\
\bar{p}(t) = p(t)+\varepsilon\varsigma[q,u,\mu,p,p_{\alpha}](t)+ o(\varepsilon) \, , \\
\bar{p}_{\alpha}(t) = p_{\alpha}(t)+\varepsilon\nu[q,u,\mu,p,p_{\alpha}](t)+ o(\varepsilon) \, , \\
\end{cases}
\end{equation}
if
\begin{multline}
\label{eq:condInv}
\left({\cal H}[\bar{q},\bar{u},\bar{\mu},\bar{p},\bar{p}_{\alpha}](\bar{t})
-\bar{p}(\bar{t})\cdot \dot{\bar{q}}(\bar{t})
-\bar{p}_{\alpha}(\bar{t}) \cdot  {_{\bar{a}}D_{\bar{t}}^\alpha}
\bar{q}(\bar{t})\right) d\bar{t} \\
=\left({\cal H}[q,u,\mu,p,p_{\alpha}](t)
-p(t)\cdot \dot{q}(t)-p_{\alpha}(t)\cdot{_aD_t^\alpha} q(t)\right) dt.
\end{multline}
\end{definition}

Next theorem provides an extension of Noether's theorem to the
wider fractional context of optimal control problems
with classical and Riemann--Liouville derivatives.

\begin{theorem}(Noether's theorem in Hamiltonian form
for optimal control problems with classical and Riemann--Liouville derivatives).
\label{thm:mainResult}
If \eqref{eq:COA}--\eqref{eq:sitRL} is variationally invariant,
in the sense of Definition~\ref{def:inv:gt1},
and functions $\xi$ and $p_{\alpha}$
satisfy condition $(\mathcal{C})$ of Theorem~\ref{trfo}, then
\begin{multline}
\label{eq:tndf:CO}
\frac{d}{dt}\Biggl[-\xi\cdot p
-\sum_{r=0}^{\infty}\Bigl((-1)^{r}p_{\alpha}^{(r)}
\cdot {_aI_t}^{r+1-\alpha}\xi\\
+\xi^{(r)} \cdot {_tI_b}^{r+1-\alpha}p_{\alpha}\Bigr)
+ \tau\left({\cal H}-(1-\alpha)p_{\alpha}
\cdot{_{a}D_{t}^{\alpha}}q\right)\Biggr]= 0
\end{multline}
along any Pontryagin extremal with classical
and Riemann--Liouville derivatives (Definition~\ref{PR}).
\end{theorem}

\begin{proof}
The fractional conservation law
\eqref{eq:tndf:CO}
is obtained by applying Theorem~\ref{theo:tndf}
to the equivalent functional \eqref{eq:COA1}.
\end{proof}

Theorem~\ref{thm:mainResult} gives an
interesting result for autonomous fractional problems.
Let us consider an autonomous fractional optimal control
problem, \textrm{i.e.}, the case when functions $L$,
$\varphi$ and $\rho$ of \eqref{eq:COA}--\eqref{eq:sitRL}
do not depend explicitly on the independent variable:
\begin{gather}
I[q(\cdot),u(\cdot),\mu(\cdot)] =\int_a^b L\left(q(t),u(t),\mu(t)\right) dt
\longrightarrow \min \, , \label{eq:FOCP:CO}\\
\dot{q}(t)=\varphi(q(t),u(t))\label{eq:FOCP:CO10}\,,\\
_aD_t^\alpha
q(t)=\rho\left(q(t),\mu(t)\right)\,. \label{eq:45}
\end{gather}

\begin{corollary}
\label{cor:FOCP:CO}
For the autonomous fractional problem
\eqref{eq:FOCP:CO}--\eqref{eq:45} one has
\begin{multline}
\label{eq:frac:eng}
\frac{d}{dt}\Biggl[ {\cal H}(q(t), u(t),\mu(t),p(t),p_{\alpha}(t))\\
-(1-\alpha)p_{\alpha}(t)\cdot{_{a}D_{t}^{\alpha}}q(t) \Biggr]= 0
\end{multline}
along any Pontryagin extremal with classical and Riemann--Liouville derivatives
$(q(\cdot),u(\cdot),\mu(\cdot),p(\cdot),p_{\alpha}(\cdot))$.
\end{corollary}

\begin{proof}
As the Hamiltonian ${\cal H}$ does not depend explicitly on the
independent variable $t$, we can easily see that
\eqref{eq:FOCP:CO}--\eqref{eq:45} is invariant
under translation of the time variable:
the condition of invariance \eqref{eq:condInv} is satisfied with
$\bar{t}(t) = t+\varepsilon$,
$\bar{q}(t) = q(t)$,
$\bar{u}(t) = u(t)$,
$\bar{\mu}(t)=\mu(t)$,
$\bar{p}(t) = p(t)$,
and $\bar{p}_{\alpha}(t)=p_{\alpha}(t)$. Indeed,
given that $d\bar{t} = dt$, the invariance
condition \eqref{eq:condInv} is verified if
${_{\bar{a}}D_{\bar{t}}^\alpha} \bar{q}(\bar{t}) =
{_aD_t^\alpha} q(t)$. This is true because
\begin{equation*}
\begin{split}
_{\bar{a}} & D_{\bar{t}}^\alpha \bar{q}(\bar{t}) \\
&= \frac{1}{\Gamma(1-\alpha)}\frac{d}{d\bar{t}}
\int_{\bar{a}}^{\bar{t}} (\bar{t}-\theta)^{-\alpha}\bar{q}(\theta)d\theta \\
&= \frac{1}{\Gamma(1-\alpha)}\frac{d}{dt}
\int_{a + \varepsilon}^{t+\varepsilon} (t + \varepsilon-\theta)^{-\alpha}\bar{q}(\theta)d\theta \\
&= \frac{1}{\Gamma(1-\alpha)}\frac{d}{dt}
\int_{a}^{t} (t-s)^{-\alpha}\bar{q}(s + \varepsilon)ds \\
&= {_{a}D_{t}^\alpha} \bar{q}(t + \varepsilon) = {_{a}D_{t}^\alpha} \bar{q}(\bar{t}) \\
&= {_{a}D_{t}^\alpha} q(t) \, .
\end{split}
\end{equation*}
Using the notation in \eqref{eq:trf:inf}, we have
$\tau = 1$, $\xi=\varrho=\varsigma=\sigma=\nu=0$.
From Theorem~\ref{thm:mainResult} we arrive to
the intended equality \eqref{eq:frac:eng}.
\end{proof}

The Corollary~\ref{cor:FOCP:CO} shows that in contrast with
the classical autonomous problem of optimal control,
for \eqref{eq:FOCP:CO}--\eqref{eq:45}
the Hamiltonian ${\cal H}$ does not define a conservation law.
Instead of the classical equality
$\frac{d}{dt}\left({\cal H}\right)=0$, we have
\begin{equation}
\label{eq:ConsHam:alpha}
\frac{d}{dt} \left[
{\cal H} + \left(\alpha-1\right) p_{\alpha}
\cdot {_aD_t^\alpha} q \right] = 0 \, ,
\end{equation}
\textrm{i.e.}, conservation of the Hamiltonian ${\cal H}$
plus a quantity that depends on the fractional order $\alpha$ of differentiation.
This seems to be explained by violation of the homogeneity
of space-time caused by the fractional
derivatives, when $\alpha\neq 1$. If $\alpha=1$, then
we obtain from \eqref{eq:ConsHam:alpha} the classical result:
the Hamiltonian ${\cal H}$ is preserved along all the
Pontryagin extremals.


\section{ILLUSTRATIVE EXAMPLE}
\label{exa}

In this section  we consider an
example where the fractional Lagrangian
and velocity vectors do not depend explicitly on the
independent variable $t$ (autonomous case).
We use our Corollary~\ref{cor:FOCP:CO} to
establish a conservation law.

\begin{example}
Consider the fractional optimal control problem
\begin{gather*}
 I[q(\cdot)] = \frac{1}{2}\int_0^1 \left[q^2(t)+
u^2(t)+\mu^2(t)\right] dt
\longrightarrow \min,\\
\dot{q}(t)=-q(t)+u(t),\\
_{0}D_{1}^{\alpha}q(t)=-q(t)+\mu(t).
\end{gather*}
The Hamiltonian ${\cal H}$ defined by \eqref{eq:HL}
takes the following form:
$${\cal H}=\frac{1}{2}\left(q^2+
u^2+\mu^2\right)+p(-q+u)+p_{\alpha}(-q+\mu).$$
It follows from our Corollary~\ref{cor:FOCP:CO} that
\begin{equation}
\label{eq:cons:Energ:Ex2}
\frac{d}{dt}\left[
\frac{1}{2}\left(q^2+u^2+\mu^2\right)+p(-q+u)
+\alpha p_{\alpha}(-q+\mu) \right] = 0
\end{equation}
along any fractional Pontryagin extremal of the problem.
\end{example}

For $\alpha = 1$ one has $u=\mu$ and $p=p_{\alpha}$,
and the conservation law \eqref{eq:cons:Energ:Ex2} can
be interpreted as \emph{conservation of energy}.


\section*{ACKNOWLEDGMENT}

\small

This work was supported by {\it FEDER} funds through
{\it COMPETE} --- Operational Programme Factors of Competitiveness
(``Programa Operacional Factores de Competitividade'')
and by Portuguese funds through the
{\it Center for Research and Development
in Mathematics and Applications} (University of Aveiro)
and the Portuguese Foundation for Science and Technology
(``FCT --- Funda\c{c}\~{a}o para a Ci\^{e}ncia e a Tecnologia''),
within project PEst-C/MAT/UI4106/2011
with COMPETE number FCOMP-01-0124-FEDER-022690.
The first author was also supported by the post-doc fellowship
SFRH/BPD/51455/2011 from FCT, program {\it Ci\^{e}ncia Global};
the second author by the Portugal--Austin (USA) project
UTAustin/MAT/0057/2008.



\bigskip

{\bf This is a preprint of a paper whose final and definite form will be published in:
51st IEEE Conference on Decision and Control, December 10-13, 2012,
Maui, Hawaii, USA. Article Source/Identifier: PLZ-CDC12.1832.45c07804.
Submitted 08-March-2012; accepted 17-July-2012.}



\begin{thebibliography}{99}

\bibitem{AGRA6}
O. P. Agrawal,
Generalized Euler-Lagrange equations and transversality
conditions for FVPs in terms of the Caputo derivative,
J. Vib. Control {\bf 13} (2007), no.~9-10, 1217--1237.

\bibitem{CD:Agrawal:2007}
O. P. Agrawal,
Fractional variational calculus in terms of Riesz fractional derivatives,
J. Phys. A {\bf 40} (2007), no.~24, 6287--6303.

\bibitem{CD:BalAv:2004}
D. Baleanu\ and\ T. Avkar,
Lagrangians with linear velocities within Riemann-Liouville fractional derivatives,
Nuovo Cimento, {\bf 119} (2004), 73--79.

\bibitem{book:Baleanu}
D. Baleanu, K. Diethelm, E. Scalas\ and\ J. J. Trujillo,
Fractional calculus: models and numerical methods,
World Scientific Publishing, Singapore, 2012.

\bibitem{CD:BoCrGr}
L. Bourdin, J. Cresson\ and\ I. Greff,
A continuous/discrete fractional Noether's theorem,
preprint, 2012, arXiv:1203.1206

\bibitem{Cresson}
J. Cresson,
Fractional embedding of differential operators and Lagrangian systems,
J. Math. Phys. {\bf 48} (2007), no.~3, 033504, 34~pp.

\bibitem{CD:Djukic:1972}
D. S. Djuki\'{c}
Noether's theorem for optimum control systems,
Internat. J. Control (1) {\bf 18} (1973), 667--672.

\bibitem{MR2349759}
R. A. El-Nabulsi\ and\ D. F. M. Torres,
Necessary optimality conditions for fractional action-like
integrals of variational calculus with Riemann-Liouville
derivatives of order $(\alpha,\beta)$,
Math. Methods Appl. Sci. {\bf 30} (2007), no.~15, 1931--1939.
{\tt arXiv:math-ph/0702099}

\bibitem{MR2421931}
R. A. El-Nabulsi\ and\ D. F. M. Torres,
Fractional actionlike variational problems,
J. Math. Phys. {\bf 49} (2008), no.~5, 053521, 7~pp.
{\tt arXiv:0804.4500}

\bibitem{MR2351636}
G. S. F. Frederico\ and\ D. F. M. Torres,
Conservation laws for invariant functionals containing
compositions, Appl. Anal. {\bf 86} (2007), no.~9, 1117--1126.
{\tt arXiv:0704.0949}

\bibitem{MR2405377}
G. S. F. Frederico\ and\ D. F. M. Torres,
Non-conservative Noether's theorem for fractional action-like
variational problems with intrinsic and observer times,
Int. J. Ecol. Econ. Stat. {\bf 9} (2007), no.~F07, 74--82.
{\tt arXiv:0711.0645}

\bibitem{CD:FredericoTorres:2007}
G. S. F. Frederico\ and\ D. F. M. Torres,
A formulation of Noether's theorem for fractional problems
of the calculus of variations,
J. Math. Anal. Appl. {\bf 334} (2007), no.~2, 834--846.
{\tt arXiv:math/0701187}

\bibitem{Caputo}
G. S. F. Frederico\ and\ D. F. M. Torres,
Fractional optimal control in the sense of Caputo
and the fractional Noether's theorem,
Int. Math. Forum {\bf 3} (2008), no.~9-12, 479--493.
{\tt arXiv:0712.1844}

\bibitem{CD:FredericoTorres:2006}
G. S. F. Frederico\ and\ D. F. M. Torres,
Fractional conservation laws in optimal control theory,
Nonlinear Dynam. {\bf 53} (2008), no.~3, 215--222.
{\tt arXiv:0711.0609}

\bibitem{CD:FredericoTorres:2010}
G. S. F. Frederico\ and\ D. F. M. Torres,
Fractional Noether's theorem in the Riesz-Caputo sense,
Appl. Math. Comput. {\bf 217} (2010), no. 3, 1023–-1033.
{\tt arXiv:1001.4507}

\bibitem{CD:Gel:1963}
I. M. Gelfand\ and\ S. V. Fomin,
{\it Calculus of variations}, Prentice-Hall, 1963.

\bibitem{MR2905862}
Md.\ M. Hasan, X. W. Tangpong\ and\ O. P. Agrawal,
A formulation and numerical scheme for fractional optimal control
of cylindrical structures subjected to general initial conditions,
in {\it Fractional dynamics and control}, 3--17, Springer, New York, 2012.

\bibitem{CD:Hilfer:2000}
R. Hilfer,
{\it Applications of fractional calculus in physics},
World Sci. Publishing, River Edge, NJ, 2000.

\bibitem{CD:Jost:1998}
J. Jost\ and\ X. Li-Jost,
{\it Calculus of variations},
Cambridge Univ. Press, Cambridge, 1998.

\bibitem{Kilbas}
A. A. Kilbas, H. M. Srivastava\ and\ J. J. Trujillo,
{\it Theory and applications of fractional differential equations},
Elsevier, Amsterdam, 2006.

\bibitem{Logan:b}
J. D. Logan,
{\it Applied mathematics},
Wiley, New York, 1987.

\bibitem{book:frac}
A. B. Malinowska\ and\ D. F. M. Torres,
{\it Introduction to the fractional calculus of variations},
Imperial College Press, London \&
World Scientific Publishing, Singapore, 2012.

\bibitem{CD:MilRos:1993}
K. S. Miller\ and\ B. Ross,
{\it An introduction to the fractional calculus
and fractional differential equations},
Wiley, New York, 1993.

\bibitem{CD:Muslih:2005}
S. I. Muslih\ and\ D. Baleanu,
Hamiltonian formulation of systems with linear velocities within
Riemann-Liouville fractional derivatives,
J. Math. Anal. Appl. {\bf 304} (2005), no.~2, 599--606.

\bibitem{Noether:1971}
E. Noether,
Invariant variation problems,
Transport Theory Statist. Phys. {\bf 3} (1971), no.~1, 186--207.

\bibitem{CD:Kosmann:2004}
E. Noether, Y. Kosmann-Schwarzbach\ and\ L. Meersseman,
{\it Les th\'{e}or\`{e}mes de Noether: invariance et lois de conservation au $XX^{e}$ si\`{e}cle},
Palaiseau (Essonne), Les \'{E}ditions de l'\'{E}cole Polytechnique, 2004.

\bibitem{FVC_Gen}
T. Odzijewicz, A. B. Malinowska and D. F. M. Torres,
Generalized fractional calculus with applications to the calculus of variations,
Comput. Math. Appl., 2012, DOI: 10.1016/j.camwa.2012.01.073.
{\tt arXiv:1201.5747}

\bibitem{MyID:227}
T. Odzijewicz, A. B. Malinowska and D. F. M. Torres,
Fractional calculus of variations in terms of a generalized fractional integral with applications to Physics,
Abstr. Appl. Anal. {\bf 2012} (2012), Article ID 871912, 24~pp.
{\tt arXiv:1203.1961}

\bibitem{MR2768178}
M. D. Ortigueira,
{\it Fractional calculus for scientists and engineers},
Lecture Notes in Electrical Engineering, 84, Springer, Dordrecht, 2011.

\bibitem{CD:Podlubny:1999}
I. Podlubny,
{\it Fractional differential equations},
Academic Press, San Diego, CA, 1999.

\bibitem{CD:Riewe:1997}
F. Riewe,
Mechanics with fractional derivatives,
Phys. Rev. E (3) {\bf 55} (1997), no.~3, part B, 3581--3592.

\bibitem{CD:SaKiMa:1993}
S. G. Samko, A. A. Kilbas\ and\ O. I. Marichev,
{\it Fractional integrals and derivatives},
Translated from the 1987 Russian original,
Gordon and Breach, Yverdon, 1993.

\bibitem{CD:JMS:Torres:2002a}
D. F. M. Torres,
On the Noether theorem for optimal control,
Eur. J. Control {\bf 8} (2002), no.~1, 56--63.

\bibitem{CD:Bruce:2004}
B. van Brunt,
{\it The calculus of variations},
Universitext, Springer, New York, 2004.

\bibitem{MR2895863}
S. A. Yousefi, A. Lotfi\ and\ M. Dehghan,
The use of a Legendre multiwavelet collocation method
for solving the fractional optimal control problems,
J. Vib. Control {\bf 17} (2011), no.~13, 2059--2065.

\end{thebibliography}
\end{document}